\documentclass[12pt]{article}
\usepackage{authblk}
\usepackage{amsmath}
\usepackage{enumerate}
\usepackage{multirow}
\usepackage{graphicx}
\usepackage{amssymb}
\usepackage{amsthm}
\usepackage{paralist}
\usepackage[misc]{ifsym}
\usepackage[colorlinks=true]{hyperref}
\hypersetup{urlcolor=blue, citecolor=red}
\allowdisplaybreaks
\usepackage{setspace}
\usepackage{mathrsfs}
\usepackage{CJKutf8}
\usepackage{caption}
\usepackage{hyperref} 
\usepackage{tikz}
\usepackage{graphicx}

\newtheorem{thm}{Theorem}[section]
\newtheorem{lem}{Lemma}[section]

\newtheorem{prop}{Proposition}[section]

\theoremstyle{definition}
\newtheorem{defn}{Definition}[section]
\newtheorem{ex}{Example}[section]
\newtheorem{rem}{Remark}[section]
\numberwithin{equation}{section}

\renewcommand{\theequation}{\thesection.\arabic{equation}}

\newcommand{\Rmnum}[1]{\expandafter\@slowromancap\romannumeral #1@}

\makeatother
\title{Dynamics  near a class of nonhyperbolic fixed points}
\author{Meihua Jin, Shihao Meng, Yunhua Zhou \\\small {College of Mathematics and Statistics, Chongqing University,
Chongqing 401331, China}
}
\date{ }
\begin{document}
\baselineskip20pt
\maketitle
\renewcommand{\theequation}{\arabic{section}.\arabic{equation}}
\catcode`@=11 \@addtoreset{equation}{section} \catcode`@=12
\begin{abstract}

In this paper, we investigate some dynamical properties near a  nonhyperbolic  fixed point.  Under some conditions on the higher nonlinear terms, we  establish a stable manifold theorem and a degenerate Hartman theorem. Furthermore, the finite shadowing property also be discussed.

\end{abstract}

\section{Introduction}\label{Section 1}

In dynamical system theory, the local properties of fixed points are crucial for understanding the system's dynamic behavior. Although much research has focused on hyperbolic fixed points, the study of nonhyperbolic fixed points remains challenging. The behavior of nonhyperbolic fixed points is complex, particularly in terms of stability and trajectory structure. How to reveal the local properties of such systems has become an important research topic.

The stable manifold theorem is a fundamental tool for describing the trajectories near hyperbolic fixed points. At the turn of the 19th and 20th centuries, during the study of the three-body problem, Poincar\'{e} was the first to recognize the importance of stable and unstable manifolds. Through geometric intuition, he described the structure of trajectories near fixed points and introduced the concepts of stable and unstable manifolds. In 1930, Perron\cite{OP} proved the existence of local stable manifolds. The proof of the stable manifold theorem has two main methods\cite{S}: Hadamard's (1901) graphical transformation method and Perron's (1929) parameter variation method. After decades of development, the stable manifold theorem has become an indispensable part of dynamical system analysis and has found widespread applications in hyperbolic systems, nonuniformly hyperbolic systems\cite{BP}, and partially hyperbolic systems\cite{KA}. Recently, the stable and unstable manifolds of certain non-hyperbolic fixed points have also been studied (e.g., see \cite{BFM,BFM2,ZZ}).

The classic Hartman Theorem (or Grobman-Hartman Theorem) asserts that a diffeomorphism locally conjugate to its linear part at a hyperbolic fixed point. This linearization result was independently obtained by Grobman\cite{G,G2} and Hartman\cite{PH,PH2} in finite-dimensional settings and later extended to Banach spaces by Palis\cite{P2} and Pugh\cite{PC}. The  Hartman Theorem plays a significant role in dynamical systems and differential equations. More detailed information about this important theorem and its applications can be found in relevant literature\cite{H,PM,KH,SM}. By replacing the hyperbolicity assumption at the fixed point with generalized hyperbolicity, the Hartman Theorem has been further generalized (e.g., see\cite{Mo}, \cite{NA}). Specifically, by imposing certain conditions, the behavior of a two-dimensional diffeomorphism near a nonhyperbolic fixed point can be topologically equivalent to a simple decoupled map, a process known as the degenerate Hartman theorem\cite{Mo}.

The shadowing property of dynamical systems (diffeomorphisms or flows) has been thoroughly studied (see related monographs and surveys\cite{PK,PYS,T}). This property indicates that for approximate trajectories (i.e., pseudo-orbits), there exist exact trajectories of the system nearby. In most cases, standard methods can prove that diffeomorphisms near hyperbolic fixed points have shadowing properties, and this property satisfies Lipschitz conditions\cite{PK}. Additionally, Quasi-shadowing for partially hyperbolic diffeomorphisms and the shadowing lemma for nonuniformly hyperbolic systems have been well established (\cite{HZ,KH,Po}). Sufficient conditions for a diffeomorphism $f$ to have a finite tracking property on $X$ near non-hyperbolic fixed points have also been provided\cite{PP,PP2}.

However, for non-hyperbolic fixed points, especially degenerate ones, a unified theoretical framework is still lacking.
This paper aims to investigate a class of local dynamical properties of nonhyperbolic fixed points, with a focus on the stable manifold theorem, degenerate Hartman Theorem, and the finite shadowing property. The results of this paper extend, to some extent, the results of \cite{M}, \cite{Mo}, and \cite{PP}.

In this paper, we consider the following  map of plane
\begin{equation}\label{eq1}
\begin{array}{rcl}
F=(f, g): \ \mathbb{R}^2 & \to & \mathbb{R}^2\\
(x, y)&\mapsto &
(x-P(x,y)+X(x,y),
y+Q(x,y)+Y(x,y)),
\end{array}
\end{equation}
where $P$ and $Q$ are  homogeneous polynomials of degree $2k+1$ and $2k'+1$ respectively for some $k, k'\in \mathbb N$, $X$ and $Y$ are higher-order terms than $P$ and $Q$. Let $\mathscr{A} $,  $\mathscr{B} $,  $\mathscr{C} $, $\mathscr{D} $, $\mathscr{E} $, and $\mathscr{H} $  be the corresponding  symmetric coefficient tensors of 
\begin{align*}
    \frac{\partial P}{\partial x},\,  \frac{\partial Q}{\partial y},\, \frac{\partial P}{\partial x} + \frac{\partial P}{\partial y},\, \frac{\partial P}{\partial x} - \frac{\partial P}{\partial y},\, \frac{\partial Q}{\partial y}+ \frac{\partial Q}{\partial x},\,and \, \frac{\partial Q}{\partial y}- \frac{\partial Q}{\partial x}.
\end{align*}
See more details about the tensor in Section \ref{sec2}.

Throughout this paper, we assume that
$X(0,0)=Y(0,0)=0.$
 That is, $\mathbf{0}=(0,0)$ is a fixed point of $F.$
 
For a neighborhood $N$ of $\mathbf 0$, we define the stable and unstable sets restricted to $N$ as
$$
\begin{aligned}
W^s(F,N)&=\{(x,y)\in N:\ F^k(x,y)\in N, k\geq 0, \lim\limits_{k\to +\infty}F^k(x,y)=(0,0)\},\\
W^u(F,N)&=\{(x,y)\in N:\ F^{-k}(x,y)\in N, k\geq 0, \lim\limits_{k\to +\infty}F^{-k}(x,y)=(0,0)\}.
\end{aligned}
$$

\begin{thm}\label{thmA}
If   $\mathscr{A}   $,  $\mathscr{B} $,  $\mathscr{C} $, $\mathscr{D} $, $\mathscr{E} $, and $\mathscr{H} $  are positive definite, then there is  a neighborhood $N$ of $\mathbf 0$ and $\delta>0$ such that the stable and unstable sets $ W^s(F,N)$ and $ W^u(F,N)$  are the graphs of two Lipschitz functions $\varphi^s$ and $\varphi^u: [-\delta, \delta]\to \mathbb{R}^1$.
\end{thm}

For any $x\in \mathbb R$ closing to $0$, there is a unique point $\mathbf{p}(x)=(p_x,p_y)\in W^s$ such that $p_x=x$. Similarly, any $y\in \mathbb R$ closing to $0$, there is a unique point $\mathbf{q}(y)=(q_x,q_y)\in W^u$ such that $q_y=y$. Define
$$F_1(x)=\pi_1\circ F(\mathbf{p}(x)),\ F_2(y)=\pi_2\circ F(\mathbf{q}(y)).$$
The following is a version of Hartman theorem.
\begin{thm}\label{thmB}
Suppose that   $\mathscr{A}   $,  $\mathscr{B} $,  $\mathscr{C} $, $\mathscr{D} $, $\mathscr{E} $, and $\mathscr{H} $  are positive definite. There are two neighborhoods $N$ and $N'$ of $(0, 0)$ such that $F|_N$ is topologically conjugate to $F'|_{N'}$, where  $F'=F_1\times F_2$.
\end{thm}

Let $f$ be a be a homeomorphism of a metric space $(M, d)$. For $\delta>0$ and $m\in \mathbb{N}$,
we say that the set \{$p_k\in M : 0\leq k \leq m$\} is a $\delta$-pseudoorbit of $f$ if
             $$d\left(p_{k+1}, f(p_k)\right) < \delta, \quad 0\leq k \leq m-1.$$
A finite pseudoorbit \{$p_k\in M : 0\leq k \leq m-1$\} is $\epsilon$-shadowed by a point $q$ if
             $$d(f^k(q), p_k) < \epsilon,  \quad 0\leq k \leq m.$$
 We say that $f$ has the {\em finite shadowing property} in a set $K\in M$, if for any $\epsilon>0$ we there is
 $\delta > 0$ such that any finite $\delta$-pseudoorbit \{$p_k\in K : 0\leq k \leq m-1$\} of $f$ can be $\epsilon$-shadowed by some point.

\begin{thm}\label{thmC}
If  $\mathscr{A}   $,  $\mathscr{B} $,  $\mathscr{C} $, $\mathscr{D} $, $\mathscr{E} $, and $\mathscr{H} $  are positive definite,  Then there is an neighborhoods $K$  of $(0, 0)$ such that $F$ has the finite shadowing property in the set $K$.
\end{thm}

\begin{rem}
If $P=x^{2n+1}$ and $Q=y^{2m+1}$ for some $n,m\in \mathbb{N}$, (\ref{eq1}) is the main example in \cite{PP}. So, Theorem \ref{thmC} extends some result (e.g., Example 2) of \cite{PP} to more general situations.
\end{rem}
\section{Preliminaries }\label{sec2}
We will recall some basic facts on tensors in this section (e.g., see \cite{QL,QCC}).
A real {\em tensor} $\mathscr{A}=\left(a_{i_1 \ldots i_m}\right)$  is a hypermatrix or a tentrix, which  represents a multi-array of entries $a_{i_1 \ldots i_m} \in \mathbb{R}$, where $i_j=1, \ldots, n_j$ for $j=1, \ldots, m$. Here, $m$ is called the order of tensor $\mathscr{A}$ and $\left(n_1, \ldots, n_m\right)$ is the dimension of $\mathscr{A}$. When $n=n_1=\cdots=n_m, \mathscr{A}$ is called an $m$th order $n$-dimensional tensor. The set of all $m$ th order $n$-dimensional real tensors is denoted as $T_{m, n}$.

For any tensor $\mathscr{A}=\left(a_{i_1 \ldots i_m}\right) \in T_{m, n}$, if its entries $a_{i_1 \ldots i_m}$ 's are invariant under any permutation of its indices, then $\mathscr{A}$ is called a {\em symmetric tensor}. The set of all $m$th order $n$-dimensional real symmetric tensors is denoted as $S_{m, n}$.

Given $\mathscr{A}=\left(a_{i_1 \ldots i_m}\right) \in T_{m, n}$ and $\mathbf{x}=(x_1, \cdots, x_n)\in \mathbb{R}^n$, define the product
$$
\mathscr{A} \mathbf{x}^m := \sum_{i_1, \ldots, i_m=1}^n a_{i_1 \cdots i_m} x_{i_1} \cdots x_{i_m}\in\mathbb{R}.
$$

\begin{ex}\label{ex3}
(\cite{QL}) (1) Suppose that a function $f: \mathbb{R}^n \rightarrow \mathbb{R}$ has continuous $m$th order derivatives. Then its $m$th order derivative $\nabla^{(m)} f(\mathbf{x})$ at any $\mathbf{x} \in \mathbb{R}^n$ is an $m$th order $n$-dimensional real symmetric tensor.

(2) For any  multi-variate homogeneous polynomial form
$$
\sum_{i_1, \ldots, i_m=1}^n a_{i_1 \ldots i_m} x_{i_1} \ldots x_{i_m},
$$
denote its coefficient tensor by $\mathscr{A}=$ $\left(a_{i_1 \ldots i_m}\right) \in T_{m, n}$. Then there is a unique symmetric tensor $\mathscr{B}=\left(b_{i_1 \ldots i_m}\right) \in S_{m, n}$ such that
$$
\sum_{i_1, \ldots, i_m=1}^n a_{i_1 \ldots i_m} x_{i_1} \ldots x_{i_m} \equiv \sum_{i_1, \ldots, i_m=1}^n b_{i_1 \ldots i_m} x_{i_1} \ldots x_{i_m} .
$$
We call $\mathscr{B}$ the symmetrization of $\mathscr{A}$, denoted as $\mathscr{B}=\operatorname{Sym}(\mathscr{A})$.
\end{ex}

An $m$th order $n$-dimensional real tensor $\mathscr{A}=\left(a_{i_1 \cdots i_n}\right) \in T_{m, n}$ is said to be {\em positive definite}  if
$$
\mathscr{A} \mathbf{x}^m> 0, \ \forall \mathbf{x} \in \mathbb{R}^n,\mathbf{x} \neq 0 .
$$
We call an multi-variate homogeneous polynomial is positive definite if its symmetric coefficient tensor is positive definite.

Obviously, for a positive definite tensor $\mathscr{A}  \in T_{m, n}$, the  order $m$ should be even.  For positive definite tensors, the symmetry is not necessarily required, and for any nonsymmetric tensor $\mathscr{A}$, its positive definiteness can be fully characterized by that of its symmetrization $\operatorname{Sym}(\mathscr{A})$ as stated in Example \ref{ex3} above.

By using the Z-eigenvalues of a tensor, the positive definiteness has the following equivalent characterizations.

\begin{lem}\label{lem4}
(see Theorem 2.18 of \cite{QL})
Suppose that $\mathscr{A}=\left(a_{i_1 \cdots i_m}\right) \in S_{m, n}$. When $m$ is even, $\mathscr{A}$ is positive definite   if and only if
$$
\begin{aligned}
& \lambda_{Z \min }(\mathscr{A})=\min\left\{ \mathscr{A} \mathbf{x}^m:\  \sum_{i=1}^n x_i^2=1, \mathbf{x} \in \mathbb{R}^n \right\}>0.\\
\end{aligned}
$$
\end{lem}

By the above Lemma \ref{lem4}, one can easily get the following
\begin{lem}\label{lem5}
Let  $\mathscr{A}=\left(a_{i_1 \cdots i_m}\right) \in S_{m, n}$ be a positive definite. There is $\varepsilon>0$ such that if $\mathscr{B}=\left(b_{i_1 \cdots i_m}\right) \in S_{m, n}$  satisfies
$|a_{i_1 \cdots i_m}-b_{i_1 \cdots i_m}|<\varepsilon, \forall 1\leq i_1, \cdots, i_m\leq n$, then
$\mathscr{B}$ is positive definite.
\end{lem}

\begin{lem}\label{lem6}
$\frac{\partial P}{\partial x} $ and $ \frac{\partial P}{\partial x} \pm \frac{\partial P}{\partial y}$ are positive definite if and only if   $ \frac{\partial P}{\partial x} +\lambda \frac{\partial P}{\partial y}$ is positive definite for any $\lambda\in [-1,1]$.
\end{lem}

\begin{proof}
We firstly suppose that $\frac{\partial P}{\partial x} $ and $ \frac{\partial P}{\partial x} \pm \frac{\partial P}{\partial y}$ are positive definite. It is obviously that  $ \frac{\partial P}{\partial x} +\lambda \frac{\partial P}{\partial y}$ is positive definite for $\lambda=0, \pm 1$. For $\lambda\in (0,1)$, noting that $ \frac{\partial P}{\partial x} +\lambda \frac{\partial P}{\partial y}=  (1-\lambda)\frac{\partial P}{\partial x}+ \lambda( \frac{\partial P}{\partial x} + \frac{\partial P}{\partial y})$,  the positive definiteness of $\frac{\partial P}{\partial x} $ and $ \frac{\partial P}{\partial x} +\frac{\partial P}{\partial y}$ implies that $ \frac{\partial P}{\partial x} +\lambda \frac{\partial P}{\partial y}$ is positive definite. The case of $\lambda\in (-1,0)$ can be proved similarly.
\end{proof}

Similarly to Lemma \ref{lem6} we have
\begin{rem}\label{rem3}
$\frac{\partial Q}{\partial y} $ and $ \frac{\partial Q}{\partial y} \pm \frac{\partial Q}{\partial x}$ are positive definite if and only if   $ \frac{\partial  Q}{\partial y} +\lambda \frac{\partial Q}{\partial x}$ is positive definite for any $\lambda\in [-1,1]$.
\end{rem}

Moreover, if   $\mathscr{A}   $,  $\mathscr{B} $,  $\mathscr{C} $, $\mathscr{D} $, $\mathscr{E} $ and $\mathscr{H} $  are positive definite, since the orders of  \( X(x,y) \) and \( Y(x,y) \) are higher than the orders of \( P(x,y) \) and \( Q(x,y) \)  respectively, there exists a neighborhood \( N \) of \( (0,0) \) such that, for any \( \mathbf{p} \in N \setminus \{(0,0)\}\), the following inequalities hold.
        \begin{align}
            0<&\left(\frac{\partial P}{\partial x} - \frac{\partial X}{\partial x} + \lambda \frac{\partial P}{\partial y} - \lambda \frac{\partial X}{\partial y}\right)(\mathbf{p}) <1,\label{al1}\\
            0<&\left(\frac{\partial Q}{\partial y} + \frac{\partial Y}{\partial y} + \lambda \frac{\partial Q}{\partial x} + \lambda \frac{\partial Y}{\partial x}\right)(\mathbf{p})<1.\label{al2}
       \end{align}

\section{Stable manifolds}

Let $\alpha$ and $\delta$ be positive. Consider the following sets:
$$
\begin{aligned}
\mathcal{C}^s(1)&=\{(x,y)\in \mathbb{R}^2:\ |y|\leq |x|\},\\
\partial \mathcal{C}^s(1)&=\{(x,y)\in \mathbb{R}^2:\ |y|= |x|\},\\
\mathcal{C}^s(1,\delta)&=\{(x,y)\in \mathbb{R}^2:\ |y|\leq|x|\leq \delta\},\\
\mathcal{C}^u(1,\delta)&=\{(x,y)\in \mathbb{R}^2:\ |x|\leq|y|\leq \delta\},\\
\mathcal{C}^u(\alpha)&=\{(x,y)\in \mathbb{R}^2:\ |y|\geq \alpha|x|\}.
\end{aligned}
$$
For $A\subset \mathbb{R}^2$, the set of $\alpha$-vertical arcs in $A$ is
$$ V(A,\alpha)=\{\Gamma\subset A:\ \Gamma \text{ is a } C^1 \text{ arc}, T_x\Gamma \subset \mathcal{C}^u(\alpha)\}.$$

To prove the stable and unstable manifolds at $(0,0)$, we need the following lemma.
\begin{lem}\label{lem2}
There exist $\alpha>1$ and $\delta>0$ such that:

(1) for any $\mathbf{p}=(p_x,p_y)\in \mathcal{C}^s(1,\delta)\setminus\{(0,0)\}$, we have that $0 < f(\mathbf{p}) < p_x$ if $p_x > 0$, and $p_x < f(\mathbf{p}) < 0$ if $p_x < 0$;

(2) $F(\partial \mathcal{C}^s(1,\delta))\cap \mathcal{C}^s(1)=(0,0)$;

(3) $D_{\mathbf{p}}F (\mathcal{C}^u(\alpha))\subset \mathcal{C}^u(\alpha)$ for any $\mathbf{p}=(p_x,p_y)\in \mathcal{C}^s(1,\delta)\setminus\{(0,0)\}$;

(4) $|g(\mathbf{p})-g(\mathbf{q}) |\geq |  {p}_y -  {q}_y  |$,  if $\mathbf{p}=(p_x,p_y), \mathbf{q}=(q_x,q_y) \in \Gamma$ for  $\Gamma \in V(\mathcal{C}^s(1,\delta),\alpha)$.
\end{lem}

\begin{proof}(1) Let $\delta_1>0$  and  $\mathbf{p}=(p_x,p_y)\in \mathcal{C}^s(1,\delta_1)\setminus\{(0,0)\}$. Put $p_y=\lambda p_x$  for some $\lambda\in[-1,1]$. Then one has
$$
\begin{aligned}
f(\mathbf{p})&=f(\mathbf{p})-f(\mathbf{0})\\&=\int_0^{p_x}\frac{d}{dt}f( t, \lambda t )dt\\
&=\int_0^{p_x} \left[1-\left(\frac{\partial P}{\partial x} -\frac{\partial X}{\partial x} +\lambda \frac{\partial P}{\partial y}  -\lambda\frac{\partial X}{\partial y}\right)( t, \lambda t ) \right]dt.
\end{aligned}
$$
If $\delta_1$ is small enough, it follows from (\ref{al1}) that
  $0<f(\mathbf{p})< p_x$ if $p_x>0$, and   $p_x<f(\mathbf{p})< 0$ if $p_x<0$.
  
(2) Take $\delta_2\in(0,\delta_1]$. By the above conclusion (1),  it suffices to show that for any $\mathbf{q}=(q_x, q_y)\in \partial \mathcal{C}^s(1,\delta_2)\setminus\{(0,0)\}$,
\begin{equation}\label{eq250207-1}
g(\mathbf{q})> q_y\text{ if } q_y>0,\text{ and }   g(\mathbf{q})< q_y \text{ if } q_y<0.
\end{equation}
Let us note that $q_x=q_y$ as $\mathbf{q}=(q_x, q_y)\in \partial \mathcal{C}^s(1,\delta_2)$, and
$$
\begin{aligned}
g(\mathbf{q})&=g(\mathbf{q})-g(\mathbf{0})\\&=\int_0^{q_y}\frac{d}{dt}g( t, t )dt\\
&=\int_0^{q_y} \left[1+ \left(\frac{\partial Q}{\partial y} +\frac{\partial Y}{\partial y} +  \frac{\partial Q}{\partial x}  + \frac{\partial Y}{\partial x}\right)( t, t ) \right]dt.
\end{aligned}
$$
If $\delta_2$ is small enough, then (\ref{eq250207-1}) holds by (\ref{al2}).
(3)
Since  $\mathscr{A} $, $\mathscr{B} $, $\mathscr{C} $, $\mathscr{D} $, $\mathscr{E} $, and $\mathscr{H} $ are positive definite, there exists $\alpha>1$  such that the corresponding tensors of $\frac{\partial P}{\partial x} \pm \alpha\frac{\partial P}{\partial y}$ and $\alpha\frac{\partial Q}{\partial y} \pm\frac{\partial Q}{\partial x}$ are positive definite. For this $\alpha$, we may find $\delta\in(0,\delta_2]$ such that for any $\mathbf{p} \in B(\mathbf{0}, \delta )$,
\begin{equation}\label{eq4}
0<\left(\frac{\partial P}{\partial x} \pm\alpha \frac{\partial P}{\partial y}-\frac{\partial X}{\partial x} \pm\alpha\frac{\partial X}{\partial y}\right)(\mathbf{p})<1,
\end{equation}
and
\begin{equation}\label{eq5}
0<\left(\alpha\frac{\partial Q}{\partial y} \pm\frac{\partial Q}{\partial x}+\alpha\frac{\partial Y}{\partial y} \pm\frac{\partial Y}{\partial x}\right)(\mathbf{p})<1.
\end{equation}
For any $\mathbf{p}\in \mathcal{C}^s(1,\delta_1)\setminus \{(0,0)\}$. Note that
$$
\begin{aligned}
D_{\mathbf{p}}F\left(
\begin{array}{c }
1 \\
\alpha
\end{array}
\right)=&
\left(
\begin{array}{cc}
1-\frac{\partial P}{\partial x}  +\frac{\partial X}{\partial x} & -\frac{\partial P}{\partial y}  + \frac{\partial X}{\partial y}\\
\frac{\partial Q}{\partial x}  +\frac{\partial Y}{\partial x}& 1+\frac{\partial Q}{\partial y}  + \frac{\partial Y}{\partial y}
\end{array}
\right)_{\mathbf{p}}
\left(
\begin{array}{c }
1 \\
\alpha
\end{array}
\right)\\
=&
\left(
\begin{array}{c}
1-\left(\frac{\partial P}{\partial x}+\alpha\frac{\partial P}{\partial y}  -\frac{\partial X}{\partial x}- \alpha\frac{\partial X}{\partial y}\right)\\
\alpha+\left(\alpha\frac{\partial Q}{\partial y}+\frac{\partial Q}{\partial x}  + \alpha\frac{\partial Y}{\partial y}+\frac{\partial Y}{\partial x}\right)
\end{array}
\right)_{\mathbf{p}}
.
\end{aligned}
$$
Denoting $D_{\mathbf{p}}F (1,\alpha)^T=(u_1,v_1)$, by (\ref{eq4}) and (\ref{eq5}), one has $0<u_1<1$ and $v_1>\alpha$.
That is, $D_{\mathbf{p}}F (1,\alpha)^T\in \mathcal{C}^u(\alpha)$. Similarly, one can prove that
$D_{\mathbf{p}}F (-1,\alpha)^T , D_{\mathbf{p}}F (1,-\alpha)^T$ and $ D_{\mathbf{p}}F (-1,-\alpha)^T$ are all in $\mathcal{C}^u(\alpha)$. So  $D_{\mathbf{p}}F (\mathcal{C}^u(\alpha))\subset \mathcal{C}^u(\alpha)$.
(4) First, we note that the $\delta$ stated in part (3) ensures that parts (1) and (2) remain valid. Let $\mathbf{p}=(p_x,p_y), \mathbf{q}=(q_x,q_y) \in \Gamma$ for some $\Gamma \in V(\mathcal{C}^s(1,\delta),\alpha)$ and observe that  $|q_x-p_x|\le\alpha^{-1}|q_y-p_y|$ where $\alpha$ is defined as in part (3). We may write $q_x-p_x=\lambda(q_y-p_y)$ for some $\lambda\in[-\alpha^{-1},\alpha^{-1}]\subset[-1,1]$.    Then by (\ref{al2}),
$$
\begin{aligned}
&|g(\mathbf{q})-g(\mathbf{p})|\\
=& \left|\int_0^{q_y-p_y}\frac{d}{dt}g(p_x+\lambda t, p_y+ t )dt\right|\\
=& \left|\int_0^{q_y-p_y} \left[1+\left(\frac{\partial Q}{\partial y}  +\frac{\partial Y}{\partial y} +\lambda \frac{\partial Q}{\partial x}  +\lambda\frac{\partial Y}{\partial x}\right)(p_x+\lambda t, p_y+ t ) \right]dt\right|\\
>& |q_y-p_y|.
\end{aligned}
$$
This completes the proof.
\end{proof}

\begin{rem}\label{rem2}
One can similarly prove that there exist $\beta\in (0,1)$ and $\delta>0$ such that:

(1) for any $\mathbf{p}=(p_x,p_y)\in \mathcal{C}^u(1,\delta)\setminus\{(0,0)\}$, we have that $g(\mathbf{p}) > p_y$ if $p_y > 0$ and $g(\mathbf{p})<p_y$ if $p_y < 0$;

(2) $F^{-1}(\partial \mathcal{C}^u(1,\delta))\cap \mathcal{C}^u(1)=(0,0)$;

(3) $D_{\mathbf{p}}F^{-1} (\mathcal{C}^s(\beta))\subset \mathcal{C}^s(\beta)$ for any $\mathbf{p}=(p_x,p_y)\in \mathcal{C}^u(1,\delta)\setminus\{(0,0)\}$;

(4) $|\pi_1\circ F^{-1}(\mathbf{p})-\pi_1\circ F^{-1}(\mathbf{q}) |\geq |  {p}_x -  {q}_x  |$,  if $\mathbf{p}=(p_x,p_y), \mathbf{q}=(q_x,q_y) \in \Gamma$ for  $\Gamma \in H(\mathcal{C}^u(1,\delta),\beta)$, where $H(\mathcal{C}^u(1,\delta),\beta)=\{\Gamma\subset \mathcal{C}^u(1,\delta):\ \Gamma \text{ is a } C^1 \text{ arc}, T_x\Gamma \subset \mathcal{C}^s(\beta)\}.$
\end{rem}

For a subset $A\in \mathbb{R}^2$, we define the stable and unstable sets of $(0,0)$ restricted to $A$ as
$$
\begin{aligned}
W^s(F,A)&=\{(x,y)\in A:\ F^k(x,y)\in A, k\geq 0, \lim\limits_{k\to +\infty}F^k(x,y)=(0,0)\},\\
W^u(F,A)&=\{(x,y)\in A:\ F^{-k}(x,y)\in A, k\geq 0, \lim\limits_{k\to +\infty}F^{-k}(x,y)=(0,0)\}.
\end{aligned}
$$
For $\alpha>0$, set
$$\widetilde{V}(\mathcal{C}^s(1,\delta),\alpha)=\{\Gamma \in V(\mathcal{C}^s(1,\delta),\alpha):\  \text{the endpoints of } \Gamma \text{ lie in } \partial \mathcal{C}^s(1)\}. $$

The following proposition is almost a stable manifold theorem. See Proposition 5 in \cite{M}.
\begin{prop}\label{prop1}
There exist $\alpha>1$ and $\delta>0$ such that for any $\Gamma\in \widetilde{V}(\mathcal{C}^s(1,\delta),\alpha)$, the intersection $\Gamma \cap W^s(F, \mathcal{C}^s(1,\delta))$ contains exactly one point.
\end{prop}

\begin{proof}
Define $\alpha>1$ and $\delta>0$ as in Lemma \ref{lem2}. Without loss of generality, we assume that the endpoints of $\Gamma\in \widetilde{V}(\mathcal{C}^s(1,\delta),\alpha)$ are $\mathbf{p}=(x_1,y_1)$ and $\mathbf{q}=(x_2, y_2)$ with
$x_1=y_1>0$ and $ x_2=-y_2>0$.
By Lemma \ref{lem2},
 $$F^k(\Gamma)\cap \mathcal{C}^s(1,\delta)\in \widetilde{V}(\mathcal{C}^s(1,\delta),\alpha),\ \forall k\geq 0.$$
Set $\Gamma_0=\Gamma$ as above. Define
$\Gamma_k=F^k(\Gamma_{k-1})\cap \mathcal{C}^s(1,\delta), \ \forall k\geq 1$ and $I_k=F^{-k}(\Gamma_k), \ k\geq 0.$ Then $\{I_k\}$ is a nested sequence of nonempty compact sets and hence $\cap_{k\in \mathbb{N}}I_k\neq \emptyset.$

Now we show that $\cap_{k\in \mathbb{N}}I_k$ contains exactly one point. If there are two points $\mathbf{p'}, \mathbf{q'} \in \cap_{k\in \mathbb{N}}I_k$, then $F^k(\mathbf{p'}),F^k(\mathbf{q'})\in\mathcal{C}^s(1,\delta),\ \forall k\geq 0$. By Lemma \ref{lem2}.(1), $\lim\limits_{k\to +\infty}F^k(\mathbf{p'})=\lim\limits_{k\to +\infty}F^k(\mathbf{q'})=(0,0)$. On the other hand, by Lemma \ref{lem2}.(4), $|\pi_2(F^k(\mathbf{p'})-F^k(\mathbf{q'}))|\geq |\pi_2( \mathbf{p'} - \mathbf{q'} )|$. Therefore $\mathbf{p'}=\mathbf{q'}$.
\end{proof}

Now we give the proof of Theorem \ref{thmA}.

\begin{proof}[Proof of Theorem \ref{thmA}.]
We only prove the existence of local stable manifold. One can similarly get the local unstable manifold using the same argument replacing $F$ by $F^{-1}$.

The above Proposition \ref{prop1} concludes that a vertical arc extending across $\mathcal{C}^s(1,\delta)$ intersects the stable set $W^s(F,\mathcal{C}^s(1,\delta))$ exactly once. It then follows immediately that $W^s(F, \mathcal{C}^s(1,\delta))$ is the graph of a Lipschitz function $\varphi^s: [-\delta, \delta] \to \mathbb{R}$. The Lipschitz constant must be less than $\alpha$. Otherwise, there exists a vertical arc $\Gamma'$ such that the intersection $\Gamma' \cap W^s(F, \mathcal{C}^s(1,\delta))$ contains more than one point.

Set
$$N=\{(x,y)\in \mathbb{R}^2: |x|\leq \delta, |y|\leq \delta\}.$$
By Lemma \ref{lem2}, we have $W^s(F, \mathcal{C}^s(1,\delta))=W^s(F, N).$ So we get the local Lipschitz stable manifold at $(0,0)$.
\end{proof}

\section{Degenerate Hartman theorem}
\begin{defn}
        Let $f$ and $g$ are two continuous mappings from a topological space $X$ to a topological space $Y$. We say that $f$ is homotopy to $g$, if there exists a continuous mapping $\gamma : X \times I \to Y$ such that for any $x \in X$,

        $$\gamma(x,0) = f(x) \quad \text{and} \quad \gamma(x,1) = g(x).$$
        Here $I = [0,1]$. The mapping $\gamma$ is called a homotopy between $f$ and $g$.
\end{defn}

\begin{proof}[Proof of Theorem \ref{thmB}]
Fix a small neighborhood $N$ of $(0, 0)$.
Take $\mathbf{p}_0=(x_{s,0}, y_{s,0})\in W^s(F,N)$ and a short vertical line $\textup{VL} \subset N$ containing $\mathbf{p}_0$ such that the endpoints $\mathbf{p}_1=(x_{s,1}, y_{s,1})$ and $ \mathbf{p}_2=(x_{s,2}, y_{s,2}) $ satisfy that $x_{s,0}=x_{s,1}=x_{s,2}>0$ and $y_{s,1}<y_{s,0}<y_{s,2}$. By the above Lemma \ref{lem2}, the image $F(\textup{VL})$ is  a vertical arc.

For given $ y\in [y_{s,1}, y_{s,2}]$ and $t\in [0,1]$, define a homotopy
$\gamma_1((x_{s,0},y),t)=tF(x_{s,0},y)+(1-t)(x_{s,0},y)$ such  that $\gamma_1(\mathbf{p}_1,[0,1])$ lies beneath $W^s(F,N)$, and $\gamma_1(\mathbf{p}_2,[0,1])$ lies above $W^s(F,N)$.
We claim that for any fixed $t\in [0,1]$, the image $\gamma_1(\textup{VL},t)$ is a vertical arc. Indeed, noting that
$$\gamma_1((x_{s,0},y),t)=
\left(
\begin{array}{c}
tf(x_{s,0},y)+(1-t)x_{s,0}\\
tg(x_{s,0},y)+(1-t)y
\end{array}
\right):=
\left(
\begin{array}{c}
\phi \\
\varphi
\end{array}
\right), $$
it is enough to show  $\left|\frac{d\varphi}{d\phi} \right|> 1.$
The conclusion is obviously hold if $t=0$.
Now we assume  $t\in (0,1]$. In this case,
$$
\begin{aligned}
\frac{d\varphi}{d\phi}=\frac{ {d\varphi}/{dy}}{ {d\phi}/{dy}}
&=\frac{t \frac{\partial g}{\partial y}(x_{s,0},y)+1-t}{t\frac{\partial f}{\partial y}(x_{s,0},y)}\\
&=\frac{t\left[1+ \frac{\partial Q}{\partial y}(x_{s,0},y)+ \frac{\partial Y}{\partial y}(x_{s,0},y)\right]+1-t}{t\left[ \frac{\partial P}{\partial y}(x_{s,0},y)+ \frac{\partial X}{\partial y}(x_{s,0},y)\right]}
=\frac{  \frac{\partial Q}{\partial y}(x_{s,0},y)+ \frac{\partial Y}{\partial y}(x_{s,0},y) + \frac1t}{ \frac{\partial P}{\partial y}(x_{s,0},y)+ \frac{\partial X}{\partial y}(x_{s,0},y) }.
\end{aligned}
$$
Noting that $\frac1t>1$, and that $\frac{\partial Q}{\partial y}, \frac{\partial Y}{\partial y} , \frac{\partial P}{\partial y}$, and $\frac{\partial X}{\partial y}$ all tend to $0$ as the point tends to $(0,0)$, we get  $\left|\frac{d\varphi}{d\phi} \right|> 1$ if $(x_{s,0},y)$ is close to $(0,0)$.

Similarly, we take $\overline{\mathbf{p}}_0=(\overline{x}_{s,0}, \overline{y}_{s,0})\in W^s(F,N)$ and a short vertical line $\overline{\textup{VL}} \subset N$ containing $\overline{\mathbf{p}}_0$ such that the endpoints $\overline{\mathbf{p}}_1=(\overline{x}_{s,1}, \overline{y}_{s,1})$ and $ \overline{\mathbf{p}}_2=(\overline{x}_{s,2}, \overline{y}_{s,2}) $ satisfy that $\overline{x}_{s,0}=\overline{x}_{s,1}=\overline{x}_{s,2}<0$ and $\overline{y}_{s,1}<\overline{y}_{s,0}<\overline{y}_{s,2}$. For  $ y\in [\overline{y}_{s,1}, \overline{y}_{s,2}]$ and $t\in [0,1]$, define a homotopy
$\overline{\gamma}_1((\overline{x}_{s,0},y),t)=tF(\overline{x}_{s,0},y)+(1-t)(\overline{x}_{s,0},y)$ such that $\overline{\gamma}_1(\overline{\mathbf{p}}_1,[0,1])$ lies beneath $W^s(F,N)$ and $\overline{\gamma}_1(\overline{\mathbf{p}}_2,[0,1])$ lies above $W^s(F,N)$. Then for any fixed $t\in [0,1]$, the image $\overline{\gamma}_1(\overline{\textup{VL}},t)$ is a vertical arc.

Now, take $\mathbf{q}_0=(x_{u,0}, y_{u,0})\in W^u(F,N)$ and a short horizontal line $\textup{HL}\subset N$ containing $\mathbf{q}_0$ such that  the endpoints $\mathbf{q}_1=(x_{u,1}, y_{u,1})  $ and $ \mathbf{q}_2=(x_{u,2}, y_{u,2})  $ satisfy that $y_{u,0}=y_{u,1}=y_{u,2}>0$ and $x_{u,1}<x_{u,0}<x_{u,2}$. Then the inverse image $F^{-1}(\textup{HL})$ is a horizontal arc. For given $ x\in [x_{u,1}, x_{u,2}]$ and $r\in [0,1]$, define a homotopy $\gamma_2((x,y_{u,0}),r)=rF^{-1}(x,y_{u,0})+(1-r)(x,y_{u,0})$. It can be similarly  checked that for any fixed $r\in [0,1]$, the image $\gamma_2(\textup{HL},r)$ is a horizontal arc.

Symmetrically,  we take $\overline{\mathbf{q}}_0=(\overline{x}_{u,0}, \overline{y}_{u,0})\in W^u(F,N)$ and a short horizontal line $\overline{\textup{HL}}\subset N$ containing $\overline{\mathbf{q}}_0$ such that  the endpoints $\overline{\mathbf{q}}_1=(\overline{x}_{u,1}, \overline{y}_{u,1})  $ and $ \overline{\mathbf{q}}_2=(\overline{x}_{u,2}, \overline{y}_{u,2})  $ satisfy that  $\overline{y}_{u,0}=\overline{y}_{u,1}=\overline{y}_{u,2}<0$ and $\overline{x}_{u,1}<\overline{x}_{u,0}<\overline{x}_{u,2}$. Then the inverse image $F^{-1}(\overline{\textup{HL}})$ is a horizontal arc. For given $ { x}\in [\overline{x}_{u,1}, \overline{x}_{u,2}]$ and $r\in [0,1]$, define a homotopy $\overline{\gamma}_2((x,\overline{y}_{u,0}),r)=rF^{-1}(x,\overline{y}_{u,0})+(1-r)(x,\overline{y}_{u,0})$. Then for any fixed $r\in [0,1]$, the image $\overline{\gamma}_2(\overline{\textup{HL}},r)$ is a horizontal arc.

We claim that, in the above construction, there exist the points such that for some $m \in \mathbb{N}$,
$$
\mathbf{q}_2=F^m(\mathbf{p}_2),\ \mathbf{q}_1=F^m(\overline{\mathbf{p}}_2), \
\overline{\mathbf{q}}_2=F^m(\mathbf{p}_1), \text{ and } \overline{\mathbf{q}}_1=F^m(\overline{\mathbf{p}}_1).
$$
Indeed, for the upper endpoints $\mathbf{p}_2$ and $\mathbf{\overline{p}}_2$ of $\textup{VL}$ and $\overline{\textup{VL}}$,  respectively, suppose that after \( m \) iterations, they are mapped to the points $\mathbf{q}_2$ and $\mathbf{q}_1$, respectively. Without loss of generality, we assume that the vertical coordinate of $\mathbf{q}_2$ is greater than that of $\mathbf{q}_1$, that is $y_{u,2}>y_{u,1}$. By lowering the vertical coordinate of $\mathbf{p}_2$, we can ensure that $\mathbf{q}_1$ and $\mathbf{q}_2$ lie on the same horizontal line $\textup{HL}$.

We reduce $N$ to the region bounded by
\begin{multline*} \bigcup_{i=0}^{m-1}F^i(\gamma_1(\mathbf{p}_1,[0,1])),\  \bigcup_{i=0}^{m-1}F^i(\gamma_1(\mathbf{p}_2,[0,1])),\   \bigcup_{i=0}^{m-1}F^i(\overline{\gamma}_1(\overline{\mathbf{p}}_1,[0,1])),\  \bigcup_{i=0}^{m-1}F^i(\overline{\gamma}_1(\overline{\mathbf{p}}_2,[0,1])),\\ 
\textup{VL} ,\   \textup{HL} ,\  \overline{\textup{VL}},\ \textup{and}\ \overline{\textup{HL}}.
\end{multline*}

\begin{figure}[htbp]
\centering
\includegraphics[width=0.5\textwidth]{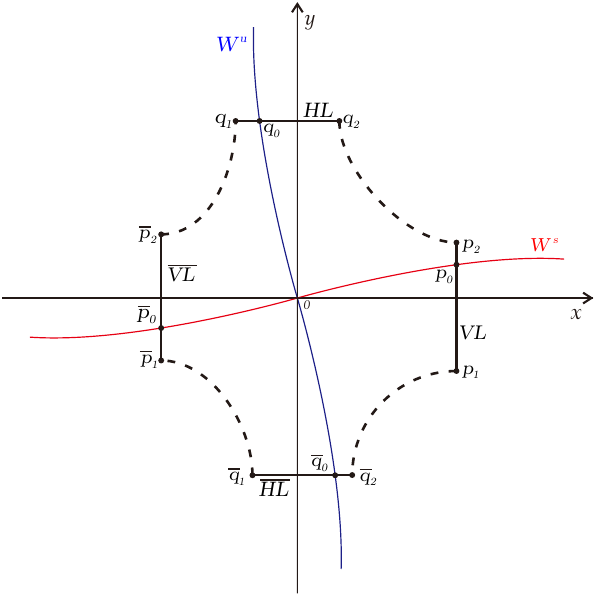}
\caption{The boundary of N. We use four dashed lines, from $\mathbf{p}_2$ to $\mathbf{q}_2$, from  $\mathbf{q}_1$ to $\overline{\mathbf{p}}_2$, from $\overline{\mathbf{p}}_1$ to  $\overline{\mathbf{q}}_1$, and from $\overline{\mathbf{q}}_2$ to  $\mathbf{p}_1$, to represent the corresponding iterated curves.} \label{figure}
\end{figure}

In the following, we will prove that the images of the vertical arcs remain vertical, and the inverse images of the horizontal arcs remain horizontal.

Since $\left|\frac{\partial P}{\partial y}\right|< \frac{\partial P}{\partial x}$ and $\left|\frac{\partial Q}{\partial x}\right|< \frac{\partial Q}{\partial y}$, there exists $\epsilon>0$ such that $\left|\frac{\partial P}{\partial y}\right|<(1-\epsilon)\frac{\partial P}{\partial x}$ and $ \left|\frac{\partial Q}{\partial x}\right|<(1-\epsilon)\frac{\partial Q}{\partial y}$.
We may assume that $N$ is sufficiently small so that the inequalities $\left|\frac{\partial X}{\partial x}\right|+\left|\frac{\partial X}{\partial y}\right|<\frac{\epsilon}{2}\left|\frac{\partial P}{\partial x}\right|$
and $\left|\frac{\partial Y}{\partial x}\right|+\left|\frac{\partial Y}{\partial y}\right|<\frac{\epsilon}{2}\left|\frac{\partial Q}{\partial y}\right|$ hold on $N$. For $(x,y)\in N$ and $\nu\in (\tau_0,1]$ where $\tau_0= 1-\frac{\epsilon}{2}$, we have
$$
\begin{aligned}
DF(x,y)\left(
\begin{array}{c }
\nu \\
1
\end{array}
\right)=&
\left(
\begin{array}{cc}
1-\frac{\partial P}{\partial x}  +\frac{\partial X}{\partial x} & -\frac{\partial P}{\partial y}  + \frac{\partial X}{\partial y}\\
\frac{\partial Q}{\partial x}  +\frac{\partial Y}{\partial x}& 1+\frac{\partial Q}{\partial y}  + \frac{\partial Y}{\partial y}
\end{array}
\right)
\left(
\begin{array}{c }
\nu \\
1
\end{array}
\right)\\
=&
\left(
\begin{array}{c}
v-\frac{\partial P}{\partial x}\nu-\frac{\partial P}{\partial y} +\frac{\partial X}{\partial x}\nu + \frac{\partial X}{\partial y}\\
1+\frac{\partial Q}{\partial x}\nu+\frac{\partial Q}{\partial y} +\frac{\partial Y}{\partial x}\nu + \frac{\partial Y}{\partial y}
\end{array}
\right).
\end{aligned}
$$
Note that
$$
\begin{aligned}
\left|\nu-\frac{\partial P}{\partial x}\nu-\frac{\partial P}{\partial y} +\frac{\partial X}{\partial x}\nu + \frac{\partial X}{\partial y}\right|&\le\left(1-\frac{\partial P}{\partial x}\right)\nu+\left|\frac{\partial P}{\partial y}\right| +\left|\frac{\partial X}{\partial x}\nu + \frac{\partial X}{\partial y}\right|\\
&\le\left(1-\frac{\partial P}{\partial x}\right)\nu+\left|\frac{\partial P}{\partial y}\right| +\left|\frac{\partial X}{\partial x}\right| + \left|\frac{\partial X}{\partial y}\right|\\
&<\left(1-\frac{\partial P}{\partial x}\right)\nu+\left(1-\frac{\epsilon}{2}\right) \frac{\partial P}{\partial x} \\
&<\nu,
\end{aligned}
$$
and
$$
\begin{aligned}
\left|1+\frac{\partial Q}{\partial x}\nu+\frac{\partial Q}{\partial y} +\frac{\partial Y}{\partial x}\nu + \frac{\partial Y}{\partial y}\right|&\ge 1+\frac{\partial Q}{\partial y}-\left|\frac{\partial Q}{\partial x}\right|\nu-\left|\frac{\partial Y}{\partial x}\right|\nu-\left|\frac{\partial Y}{\partial y}\right|\\
&\ge 1+\frac{\partial Q}{\partial y}-\left|\frac{\partial Q}{\partial x}\right|-\left|\frac{\partial Y}{\partial x}\right|-\left|\frac{\partial Y}{\partial y}\right|\\
&>1+\frac{\partial Q}{\partial y}-\left(1-\frac{\epsilon}{2}\right) \frac{\partial Q}{\partial y} \\
&>1.
\end{aligned}
$$
It follows that
$$DF(x,y)(\mathcal C^u(\tau^{-1}))\subset \mathcal C^u(\tau^{-1}),\ \forall \tau\in (\tau_0,1]. $$
In another words, we have $DF(x,y)(\mathcal C^u(\kappa))\subset \mathcal C^u(\kappa)$ for some $\kappa\in(1,\tau^{-1}_0)$.
The cone field \( C^u(\kappa) \) is invariant in the region \( N \). That is, the images of the vertical arcs remain vertical.

One can similarly prove that the inverse  images of the horizontal  arcs remain horizontal.

Observe that, together with $W^u(F,N)$, all the arcs $F^i(\gamma_1(\textup{VL},t))$ and $F^i(\overline{\gamma}_1(\overline{\textup{VL}},t))$ for $t\in[0,1]$ and $i\in\mathbb{N}$ form a (vertical) foliation $\mathcal{F}_1$ for $N$ and, together with $W^s(F,N)$,  all the arcs $F^{-j}(\gamma_2(\textup{HL},r))$ and $F^{-j}(\overline{\gamma}_2(\overline{\textup{HL}},r))$ for $r\in[0,1]$ and $j\in\mathbb{N}$ form a (horizontal) foliation $\mathcal{F}_2$ for $N$. By the above discussion, every point $(x,y)$ in $N$ can be uniquely determined by the intersection of a vertical leaf from $\mathcal{F}_1$ and a horizontal leaf from $\mathcal{F}_2$.
Thus we may write any $(x,y)\in N\setminus(W^s(F, N)\cup W^u(F,N))$ as
$$(x,y)=F^i(\widetilde{\gamma}_1(\widetilde{\textup{VL}},t))\cap F^{-j}(\widetilde{\gamma}_2(\widetilde{\textup{HL}},r)),$$
for some $t,r\in[0,1]$, $\widetilde{\textup{VL}}\in \{\textup{VL}, \overline{\textup{VL}}\}$, $\widetilde{\textup{HL}}\in \{\textup{HL}, \overline{\textup{HL}}\}$, $\widetilde{\gamma}_1\in \{{\gamma}_1, \overline{\gamma}_1 \}$, $\widetilde{\gamma}_2\in \{{\gamma}_2, \overline{\gamma}_2\}$, and  $i,j\in\mathbb{N}  $.

Now we consider the map $F'=F_1\times F_2$.  Applying the same technology as $F$,
take points $\mathbf{p}'_1$,  $\mathbf{p}'_2$,  $\overline{\mathbf{p}}'_1$,  $\overline{\mathbf{p}}'_2$,  $\mathbf{q}'_1$,  $\mathbf{q}'_2$,  $\overline{\mathbf{q}}'_1$,  $\overline{\mathbf{q}}'_2$, vertical lines $\text{VL}'$, $\overline{\text{VL}}'$ and horizontal lines $\text{HL}'$, $\overline{\text{HL}}'$ such that for some $m\in \mathbb N$,
 $$
\mathbf{q}'_2=F'^m(\mathbf{p}'_2),\ \mathbf{q}'_1=F'^m(\overline{\mathbf{p}}'_2), \
\overline{\mathbf{q}}'_2=F'^m(\mathbf{p}'_1), \text{ and } \overline{\mathbf{q}}'_1=F'^m(\overline{\mathbf{p}}'_1).
$$
Involving two homotopies $\gamma_1'$ and $\overline{\gamma}_1'$, we obtain a region $N'$ bounded by
\begin{multline*}
\bigcup_{i=0}^{m-1}F'^i(\gamma'_1(\mathbf{p}'_1,[0,1])),\ \bigcup_{i=0}^{m-1}F'^i(\gamma'_1(\mathbf{p}'_2,[0,1])),\ 
\bigcup_{i=0}^{m-1}F'^i(\overline{\gamma}'_1(\overline{\mathbf{p}}'_1,[0,1])),\ 
\bigcup_{i=0}^{m-1}F'^i(\overline{\gamma}'_1(\overline{\mathbf{p}}'_2,[0,1])),\\
\textup{VL}' ,\ \textup{HL}' ,\ \overline{\textup{VL}}',\ \textup{and}\ \overline{\textup{HL}}'.
\end{multline*}

Since the sequence $\{F_1^k(x)\}_{k\in\mathbb{N}}$, determined by a small initial value of $x>0$, is decreasing and bounded below, it converges. Similarly, $\{F_2^{-k}\}_{k\in\mathbb{N}}$ converges. In fact, both sequences converges to $0$.

Note that all the arcs $(F')^{i'}(\gamma'_1(\textup{VL}',t'))$ and $(F')^{i'}(\overline{\gamma}'_1(\overline{\textup{VL}}',t'))$ for $t'\in[0,1]$ and $i'\in\mathbb{N}$ form a (vertical) foliation $\mathcal{F}'_1$ for $N'\setminus \{(x,y)\in N': x=0\}$ and,  all the arcs $(F')^{-j'}(\gamma'_2(\textup{HL}',r'))$ and $(F')^{-j'}(\overline{\gamma}'_2(\overline{\textup{HL}}',r'))$ for $r'\in[0,1]$ and $j'\in\mathbb{N}$ form a (horizontal) foliation $\mathcal{F}'_2$ for $N'\setminus \{(x,y)\in N': y=0\}$. Every point $(x',y')$ in $N'$ with $xy\neq 0$ can be uniquely determined by the intersection of a vertical leaf from $\mathcal{F}'_1$ and a horizontal leaf from $\mathcal{F}'_2$. Thus we may write
$$(x',y')=(F')^{i'}(\widetilde{\gamma}'_1(\widetilde{\textup{VL}},t'))\cap (F')^{-j'}(\widetilde{\gamma}'_2(\widetilde{\textup{HL}},r')),$$
for some $t',r'\in[0,1]$, $\widetilde{\textup{VL}}\in \{\textup{VL}', \overline{\textup{VL}}'\}$, $\widetilde{\textup{HL}}\in \{\textup{HL}', \overline{\textup{HL}}'\}$, $\widetilde{\gamma}'_1\in \{{\gamma}'_1, \overline{\gamma}'_1 \}$, $\widetilde{\gamma}'_2\in \{{\gamma}'_2, \overline{\gamma}'_2\}$, and  $i',j'\in\mathbb{N}  $.

We claim that $(N,F)$ is topologically conjugate to $(N',F')$.

Define a map $H:N\rightarrow N'$ by
$$
\begin{aligned}
H(x,y)&=H(F^i(\widetilde{\gamma}_1(\widetilde{\textup{VL}},t))\cap F^{-j}(\widetilde{\gamma}_2(\widetilde{\textup{HL}},r)))\\
&=(F')^{i}(\widetilde{\gamma}'_1(\widetilde{\textup{VL}}',t))\cap (F')^{-j}(\widetilde{\gamma}'_2(\widetilde{\textup{HL}}',r)).
\end{aligned}
$$
Here $(x,y)=F^i(\widetilde{\gamma}_1(\widetilde{\textup{VL}},t))\cap F^{-j}(\widetilde{\gamma}_2(\widetilde{\textup{HL}},r))$ for some $t,r\in[0,1]$ and $i,j\in\mathbb{N}$. We have that
$$
\begin{aligned}
H\circ F(x,y)&=H(F(F^i(\widetilde{\gamma}_1(\widetilde{\textup{VL}},t))\cap F^{-j}(\widetilde{\gamma}_2(\widetilde{\textup{HL}},r))))\\
&=H(F^{i+1}(\widetilde{\gamma}_1(\widetilde{\textup{VL}},t))\cap F^{-j+1}(\widetilde{\gamma}_2(\widetilde{\textup{HL}},r)))\\
&=(F')^{i+1}(\widetilde{\gamma}'_1(\widetilde{\textup{VL}}',t))\cap (F')^{-j+1}(\widetilde{\gamma}'_2(\widetilde{\textup{HL}}',r))\\
&=F'((F')^{i}(\widetilde{\gamma}'_1(\widetilde{\textup{VL}}',t))\cap (F')^{-j}(\widetilde{\gamma}'_2(\widetilde{\textup{HL}}',r)))\\
&=F'\circ H(x,y).
\end{aligned}
$$
On the other hand, open sets having as boundaries the segments of $\mu$-horizontal and $\mu$-vertical curves from the foliations $\mathcal{F}_1$, $\mathcal{F}_2$, $\mathcal{F}'_1$, and $\mathcal{F}'_2$ form a base for the topology on $N$ and $N'$ respectively. Therefore $H$ is a
homeomorphism.
\end{proof}

\section{Finite shadowing property}
The following lemma is one of the main results of \cite{PP} which will be used in our proof of Theorem \ref{thmC}.

\begin{lem}\label{lem7}
(Theorem 4.1 of \cite{PP}) Let $F=(f,g):\mathbb{R}^2\to \mathbb{R}^2$ be a homeomorphism and $K$ be a compact subset of the plane.
Suppose that for any $\Delta_0>0$ there exist $\delta, \Delta>0$ such that $\delta<\Delta<\Delta_0$ and if $\mathbf{p}=(p_x, p_y) \in K$, one has that:

(1)
$$
F(S(\mathbf{p},\delta))\subset int(S(F(\mathbf{p}),\Delta))
\text{ and }
F^{-1}(S(F(\mathbf{p}),\delta))\subset int(S( \mathbf{p} ,\Delta)),
$$
where $S(\mathbf{p},\delta)=\{\mathbf{q}=(q_x, q_y)\in \mathbb{R}^2:\ |q_x-p_x|\leq \delta, |q_y-p_y|\leq \delta\}$ and $int(\cdot)$ is the interior of the set;

(2)
\begin{equation*}
    \begin{aligned}
        \left|g\left(p_x + v, p_y\right) - g\left(p_x, p_y\right)\right| &< \delta \text{ for } 0 \leq |v| \leq \Delta,\\
        \left|g\left(p_x + v, p_y + w\right) - g\left(p_x, p_y\right)\right| &> \delta \text{ for } 0 \leq |v| \leq \delta, |w| = \delta,\\
        \left|f\left(p_x + v, p_y + w\right) - f\left(p_x, p_y\right)\right| &< \delta \text{ for } (v, w) \in H(\delta),
    \end{aligned}
\end{equation*}
\raggedright where
\begin{equation*}
    H(\delta) = \{|v| \leq \delta, w = 0\} \cup \{|v| = \delta, |w| \leq \delta\}.
\end{equation*}

Then $F$ has the finite shadowing property in the set $K$.
\end{lem}

\begin{proof}[Proof of Theorem \ref{thmC}.]
(1) Noting that
$$
        DF(x,y) =
            \begin{pmatrix}
                1 - \frac{\partial P(x,y)}{\partial x} + \frac{\partial X(x,y)}{\partial x} & -\frac{\partial P(x,y)}{\partial y} + \frac{\partial X(x,y)}{\partial y} \\
                \frac{\partial Q(x,y)}{\partial x} + \frac{\partial Y(x,y)}{\partial x} & 1 + \frac{\partial Q(x,y)}{\partial y} + \frac{\partial Y(x,y)}{\partial y}
            \end{pmatrix},
$$
 for any \( \alpha > 0 \), there exists a neighborhood  \( K_1 \) of $\mathbf{0}$ such that
\begin{equation}\label{eq6}
   \| DF(x,y) \| \leq 1 + \alpha\quad \text{and} \quad \| DF^{-1}(x,y) \| \leq 1 + \alpha,\ \forall (x,y)\in K_1,
\end{equation}
where $\|\cdot \|$  is the norm of the linear operator.

Assuming \( \alpha < 1 \), let \( \Delta = 2\delta \).
Choose a subset $K_2\subset K_1$ such that $\mathbf{0}\in K_2$, $F(K_2)\subset K_1$ and $F^{-1}(K_2)\subset K_1$.
By (\ref{eq6}), we have
\begin{equation}\label{eq8}
   F(S(\mathbf{p},\delta))\subset int(S(F(\mathbf{p}),\Delta))
,\ \
F^{-1}(S(F(\mathbf{p}),\delta))\subset int(S( \mathbf{p} ,\Delta)), \ \forall \mathbf{p}\in K_2.
\end{equation}

(2)
If necessary, we further shrink
$K_1$ so that
\begin{equation}\label{eq7}
    \left| \frac{\partial Y(x,y)}{\partial x} \right| \leq \frac{1}{4}
     \text{ and } \left| \frac{\partial Q(x,y)}{\partial x} \right| \leq \frac{1}{4},\ \forall (x,y)\in K_1.
\end{equation}
Without loss of generality, we assume that
$$
S(K_2, \Delta)=\{\mathbf{q}=(q_x, q_y)\in \mathbb{R}^2:\ |q_x-p_x|\leq \delta, |q_y-p_y|\leq \delta, \mathbf{p}=(p_x, p_y)\in K_2\}\subset K_1.
$$

For any $\mathbf{p}=(p_x, p_y)\in K_2$,  $v\in \mathbb{R}$ with  \( |v| \leq \Delta = 2\delta \), by the Mean Value Theorem and (\ref{eq7}), there are $\xi_1, \xi_2\in \mathbb R$ such that
\begin{equation}\label{eq9}
\begin{array}{rcl}
|g(p_x + v, p_y) - g(p_x, p_y)|& \leq& |Q(p_x + v, p_y) - Q(p_x, p_y)| + |Y(p_x + v, p_y) - Y(p_x, p_y)|\\
& \leq & \left| \frac{\partial Q(\xi_1, p_y)}{\partial x} \right| |v| + \left| \frac{\partial Y(\xi_2, p_y)}{\partial x} \right| |v|\\
& \leq & \frac{|v|}{4}+\frac{|v|}{4}\\
& < & \delta.
\end{array}
\end{equation}

(3) Assume that $v,w\in \mathbb R$ satisfy  \( 0 \leq |v| \leq \delta \) and \( |w| = \delta \). Let \( v = \lambda w \), where \( |\lambda| \leq 1 \). Then, by (\ref{al2}) we have
\begin{equation}\label{eq10}
\begin{array}{rcl}
|g(p_x + v, p_y + w) - g(p_x, p_y)| &=& |g(p_x + \lambda w, p_y + w) - g(p_x, p_y)|\\
& = & \left| \int_0^w \frac{d}{dt} g(p_x + \lambda t, p_y + t) \, dt \right|\\
&=& \left| \int_0^w \left(1+\frac{\partial Q}{\partial y} + \frac{\partial Y}{\partial y} + \lambda \frac{\partial Q}{\partial x} + \lambda \frac{\partial Y}{\partial x}\right)(p_x + \lambda t, p_y + t) dt \right|\\
   &>&|w|=\delta.
\end{array}
\end{equation}
(4) 
Assume \( 0 \leq |w| \leq \delta \) and \( |v| = \delta \). Let \( w = \lambda v \), where \( |\lambda| \leq 1 \). Then, by (\ref{al1}) we have
\begin{equation}\label{eq11}
\begin{array}{rcl}
|f(p_x + v, p_y + w) - f(p_x, p_y)|& =& \left| \int_0^v \frac{d}{dt} f(p_x + t, p_y + \lambda t) \, dt \right|\\
&=&\left| \int_0^v \left(1-\left(\frac{\partial P}{\partial x} - \frac{\partial X}{\partial x} + \lambda \frac{\partial P}{\partial y} - \lambda \frac{\partial X}{\partial y}\right)\right)(p_x + t, p_y + \lambda t) dt\right|\\
                &<&|v| = \delta.
\end{array}
\end{equation}

Let $w=0, |v|\leq \delta$. Using (\ref{al1}) again,  we have
\begin{equation}\label{eq12}
\begin{array}{rcl}
|f(p_x + v, p_y ) - f(p_x, p_y)|& =& \left| \int_0^v \frac{d}{dt} f(p_x + t, p_y ) \, dt \right|\\
&=&\left| \int_0^v \left(1-\left(\frac{\partial P}{\partial x} - \frac{\partial X}{\partial x}\right)\right)(p_x + t, p_y) dt\right|\\
                &<&|v| \\
                &<&\delta.
\end{array}
\end{equation}
Setting $K=K_2$, by (\ref{eq8}), (\ref{eq9})-(\ref{eq12}) and Lemma \ref{lem7},we conclude that \( F  \) exhibits the finite shadowing property  on \( K \). This completes the proof.
\end{proof}
\footnotesize
\begin{spacing}{0.7}

\end{spacing}

\textit{Email address:} 202206021016@stu.cqu.edu.cn, mengshihao2023@163.com, zhouyh@cqu.edu.cn.

\end{document}